\documentclass[a4paper,12pt]{article}
\usepackage{geometry} 
\usepackage{fullpage}
\usepackage[parfill]{parskip}    % Activate to begin paragraphs with an empty line rather than an indent
\usepackage{graphicx}
\usepackage{amssymb} 
\usepackage{epstopdf}
\usepackage{amsmath}

\usepackage{amsmath,amsthm,amscd,amssymb}
\usepackage{latexsym}
\usepackage[colorlinks,citecolor=red,pagebackref,hypertexnames=false]{hyperref}
\numberwithin{equation}{section}

\theoremstyle{plain}
\newtheorem{theorem}{Theorem}[section]
\newtheorem{lemma}[theorem]{Lemma}
\newtheorem{corollary}[theorem]{Corollary}

\newtheorem{question}[theorem]{Question}

\newtheorem{definition}[theorem]{Definition}

\newtheorem{example}[theorem]{Example}

\theoremstyle{remark}
\newtheorem*{remark}{Remark}

\newtheorem{case[theorem]}{Case}

\def\supp{\hbox{supp\,}}
\def\norm#1.#2.{\lVert#1\rVert_{#2}}

  \newtheorem*{Wang}{{\bf Theorem A}}
  \newtheorem*{Fuehr}{{\bf Theorem B}}

\title{Riesz Wavelets, Tiling and Spectral Sets in LCA  Groups}
 
\author{Azita Mayeli}                         
 
\begin{document}

 \maketitle

\begin{abstract}  
This paper is devoted to the study of  geometry properties of wavelet and Riesz wavelet  sets on locally compact abelian groups. The catalyst for our research is a result by Wang (\cite{Wang2002}, Theorem 1.1) in the Euclidean wavelet theory. Here, we extend the result  to  wavelet and Riesz wavelet    collection of sets in  infinite locally compact abelian groups. 
\end{abstract}
 
 
\section{Introduction}  

Let $\varphi\in L^2(\Bbb R^n)$, $n\geq 1$. $\varphi$ is called a wavelet for $L^2(\Bbb R^n)$ if there exists countable sets $\mathcal L\subset \Bbb R^n$ and $\mathcal D\subset GL(n,\Bbb R)$ such that the  family 
\begin{align}\label{classical wavelet family}
\{ |\det D|^{1/2} \varphi(Dx-t): \ D\in \mathcal D,  t\in \mathcal L\} 
\end{align}
is an orthonormal basis for $L^2(\Bbb R^n)$. In this case, the sets $\mathcal D$ and $\mathcal L$ are called  {\it dilation} and {\it translation} sets, respectively. Wavelets are important tools in approximation theory due to their time and frequency localization property. %A wavelet is called bandlimited if its Fourier transform is supported on a bounded set in $\Bbb R^n$. 

One type of wavelets with simplest structures are those  which are supported on a   Lebesgue measurable set in $\Bbb R^n$ with finite and positive measure. The most basic one is the Haar wavelet $\varphi(x)= 1_{[0,1/2)} - 1_{[1/2, 1)}$, constructed by  Alfred Haar in 1910,  is a one-dimensional example of this kind and it serves as a prototypical wavelet. The dilation and translation sets associated to the Haar wavelet are given by $\mathcal D=\{2^m: \ m\in \Bbb Z\}$ (dyadic dilation) and $\mathcal L= \Bbb Z$.   Another type of simple wavelets are  those  whose  Fourier transforms are characteristic functions of a Lebesgue measurable domain in $\Bbb R^n$ with finite measure.  A typical example of these kind is {\it Littlewood-Paley} or {\it  Shannon} wavelet  whose Fourier transform is $1_{[-1,-1/2)\cup (1/2,1]}$. % with respect to the dyadic dilation and integer translations. 
Shannon wavelet is a  significant tool  in sampling and reconstruction of functions.   Haar wavelet and Shannon wavelet  are situated at the opposite side of each other and "somehow" related. However, in this paper we shall only focus on Shannon-type wavelets, i.e., those with Fourier transform  supported in a domain with finite measure.

 The Fourier support of the Shannon wavelet $I:=[-1,-1/2)\cup (1/2,1]$  has this nice geometry property that its dyadic dilations by $\{2^n: \ n\in \Bbb Z\}$ as well as it integer translations tile   $\Bbb R$.  The sets with the similar tiling properties have been considered in the paper by Dai and Larson \cite{DL} where the authors proved the following results: Given a measurable set $E\subset \Bbb R$ with finite and positive measure, the function $\varphi$ with $\hat \varphi= 1_E$ is a wavelet with dyadic dilation and integer translations if and only if $E$ tiles $\Bbb R$ by dyadic dilations as well as by integer translations. This result and characterization of wavelets in terms of tilings properties of their Fourier supports  was  somehow generalized by Wang \cite{Wang2002}  in $n$-dimensional Euclidean spaces, where the dilation set is any   countable  subset of  invertible matrices and  the translation set is any countable  subset  in $\Bbb R^n$. Before we recall his  result, we need several definitions. 
 
\begin{definition} A measurable subset $\Omega\subset \Bbb R^n$ with finite and positive measure is called {\it wavelet set} if for the function 
$\varphi\in L^2(\Bbb R^n)$ with $\hat\varphi= 1_\Omega$  the family (\ref{classical wavelet family}) is an orthonormal basis for $L^2(\Bbb R^n)$. \end{definition} 

These wavelets are also known as MSF wavelets (minimally supported frequency wavelets)  or in older literatures as  s-wavelets. The wavelet sets and minimally supported frequency wavelets were introduced by Fang and Wang (\cite{Fang-Wang96}) and studied exclusively by Hernandez, Wang and Weiss (\cite{HWW96,HWW97}). For a series of paper on wavelet sets see e.g., \cite{BaggettMM-wavelet sets, BB2009, DLS-wavelet sets,Olafsson_wavelet sets}  and references contained therein. 

{\it Additive and multiplicative tilings.} We say a  set $\Omega$ tiles $\Bbb R^n$ {\it additively}  with respect to a   countable set $\mathcal L\subset \Bbb R^n$ if 
$$\sum_{\ell \in \mathcal L} 1_{\Omega}(x-\ell) = 1 \quad a.e. \ \ x\in \Bbb R^n .$$
Equivalently, $\Bbb R^n  =\cup_{\ell\in \mathcal L} \Omega+\ell$, where for any distinct $\ell$ and $\ell\rq{}$, the sets $\Omega+\ell$ and $\Omega+\ell\rq{}$ are disjoint in measure,  i.e., $|\Omega+\ell \cap \Omega+\ell\rq{}| = 0$. Here, $|\Omega|$ denotes the Lebsgue measure. 
  We say $\Omega$  is a {\it multiplicative tiling set} for  $\Bbb R^n$   with respect to a collection of $n$ by $n$ invertible matrices $A$   if $\{\alpha(\Omega): \alpha\in   A\}$ is a partition for $\Bbb R^n$, i.e., 
$$\sum_{\alpha\in   A} 1_{\alpha(\Omega)}(x) = 1 \quad a.e.  \ x\in \Bbb R^n .$$
This is also equivalent to say that $\Bbb R^n = \cup_{\alpha\in   A} \alpha(\Omega)$ where for any $ \alpha \neq \alpha\rq{}$ the two sets $\alpha(\Omega)$ and $\alpha\rq{}(\Omega)$ are disjoint in measure.

The characterization of the wavelet sets in terms of  translation and multiplicative tiling was obtained by Wang  (Theorem 1.1, \cite{Wang2002}). 

\vskip.125in

\begin{Wang}\label{Wang} {\it Let $\Omega\subset \Bbb R^n$ with finite and positive Lebesgue measure, $ A\subset GL(n, \Bbb R)$  and  $\mathcal L$ be a countable subset of $\Bbb R^n$.   For $\alpha\in  A$, let $\alpha^{\tau}$ denote the transpose matrix of $\alpha$. If the set $\{\alpha^{\tau}(\Omega):   \alpha\in   A\}$ tiles $\Bbb R^n$ and the set of  exponentials $\mathcal E_{\mathcal L}:=\{e^{2\pi i l\cdot x}: \ l\in \mathcal L\}$ is an orthogonal basis for $L^2(\Omega)$, then $\Omega$ is a wavelet set with respect to the dilation set $  A$ and translation set $\mathcal L$. The converse also holds, provided that $0\in \mathcal L$. }
\end{Wang} 

Here, $l \cdot x$ denote the inner product in $\Bbb R^n$ and ``$x$\rq\rq{} is generic.

In his paper \cite{Wang2002}, Wang shows that  the additional condition $0\in \mathcal L$ can not be dropped. Indeed,   $\Omega= [-1,-1/2]\cup [1/2,1]$  is a wavelet set with respect to the dilations $M= \{\pm 2^n\}_{n\in \Bbb Z}$ and translations $\mathcal T= 2\Bbb Z +1/6$.  However,  $\Omega$  nor tiles  $\Bbb R$  multiplicatively by $M$ neither
the exponentials $\mathcal E_{\mathcal T}$ is dense in $L^2(\Omega)$. 
  The latter is true  since, for example,  the non-zero function $g(x)= e^{-\pi i x/3}$ is orthogonal to all   elements of $\mathcal E_{\mathcal T}$ with inner product in $L^2(\Omega)$.

In the classical context,    given $\Omega$ and $\Lambda$, the  $(\Omega, \Lambda)$ is called  {\it spectral pair} if the exponentials $\mathcal E_\Lambda:=\{e^{2\pi i\lambda\cdot x}: \ \lambda\in \Lambda\}$ from  an orthogonal basis for $L^2(\Omega)$. In this case, $\Omega $ is spectral and $\Lambda$ is spectra (or spectrum).  It has been shown by Fuglede  (\cite{Fug74}) that when 
 $\Lambda$ is a lattice in $\Bbb R^n$, then the existence of exponential orthogonal basis $\mathcal E_\Lambda$ for $L^2(\Omega)$  is equivalent to say that $\Omega$ tiles $\Bbb R^n$ additively  by  the  dual  lattice $\Lambda^\perp$. 
        The relation between exponential basis for spectral sets and their tiling property in general  
   has been conjectured by Fuglede in \lq{}74. The Fuglede Conjecture   
 asserts   that   given a bounded set $\Omega \subset {\Bbb R}^d$ of positive Lebesgue measure,  $L^2(\Omega)$ has an orthogonal basis of exponentials $\mathcal E_\Lambda= \{e^{2\pi i  \lambda \cdot x  } \}_{\lambda\in \Lambda}$ for some countable subset $\Lambda\subset \Bbb R^n$ if and only if $\Omega$ tiles ${\Bbb R}^d$ by additively by some translations $\mathcal T$. 
 %Fuglede proved this conjecture in the celebrated 1974 paper (\cite{Fug74}) in the case when either the tiling set $\Omega$ or the spectra $\Lambda $ is a lattice. (Recall that a lattice in$\Bbb R^n$ is the image of $\Bbb Z^n$ under an invertible $n\times n$ matrix.)
 %  In 2003, Katz, Tao  and the first listed author (\cite{IKT03}) proved that the Fuglede conjecture holds for convex planar domains. 
The Fuglede Conjecture led to considerable activity in the past three decades. 
 For example, in 2001,  Iosevich, Katz, Tao  (\cite{IKT03}) proved the conjecture for complex planar domain.  In \cite{L01} and \cite{L02}, \L aba established the conjecture for the union of two intervals on the real line and provided a series of
connections between the study of orthonormal bases and interesting problems in algebraic number theory. 
However, 
in 2004, Tao (\cite{T04}) disproved the Fuglede conjecture for dimension $5$ and higher, followed by 
  \cite{KM06} {\color{red} double check the reference}, where Kolountzakis and Matolcsi also disprove   the Fuglede Conjecture for dimensions $3$ and higher.  Yet the conjecture may still be true in several important special cases. 
  In spite of  the disproof of its general validity, the conjecture has  still generated many interesting results.  Most recently, it was proved that the Fuglede Conjecture holds in some abelian finite groups (\cite{IMP}). For the link between tiling property and existence of general bases in locally compact abelian groups see \cite{BHM-Fuglede}.

{\it $k$-additive tiling.} 
Given an integer number $k\geq 1$, we say $\Omega$  multi-tiles  $\Bbb R^d$  additively by  a multiset $\mathcal L$ with multiplicity  $k$ (or is a $k$-additive tiling) when 
$$\sum_{\ell \in \mathcal L} 1_\Omega(x-\ell) = k  \quad \ a.e. \  x\in \Bbb R^n .$$ 
By the additive multi-tiling property,  almost every point in $\Bbb R^n$ can be covered by exactly $k$-translations of $\Omega$. In dimension $d=2$, it is known that 
  every $k$-additive tiling of $\Bbb R^2$ by a convex polygon,  which  is  not  a  parallelogram,  is  a $k$-additive tiling  with  a  finite  union  of  two-dimensional lattices (\cite{Kolountzakis2001}). 
When $k=1$, then we have the tiling situation. The multi-tiling property of a bounded measurable set  is a sufficient condition for existence of exponentials set $\mathcal E_{\Lambda}$ which is  a Riesz bases for $L^2(\Omega)$.  This result was proved by Kolountzakis (\cite{K-Riesz15}, Theorem 1).

An extension of the sufficiency  part    of Theorem A has been recently proved by F\"uhr and Maus \cite{FuehrMaus15} where the authors replace the additive  tiling   by additive multi-tiling property, thus the orthogonality by Riesz bases condition. Their result is the following. 

\vskip.125in 

\begin{Fuehr}\label{RB-Hartmut}(\cite{FuehrMaus15}, Theorem 1.2) {\it Let $\Omega\subset \Bbb R^n$ be a bounded and measurable set with positive measure. Assume that $\Omega$ tiles $\Bbb R^n$ multiplicatively by a countable subset of $GL(n, \Bbb R)$ and multi-tiles $\Bbb R^n$ additively  by a lattice  with multiplicity $k$ ($k\in\Bbb N$). Then $\Omega$ is a Riesz wavelet set in $\Bbb R^n$, i.e, the system (\ref{classical wavelet family})  for $\varphi$ with $\hat \varphi=1_\Omega$ is a Riesz basis for $L^2(\Bbb R^n)$. } \end{Fuehr}

%The proof of  Theorem B is an adaptation of the argument employed in the proof of Theorem A for orthonormal wavelet sets along with a result due to Kolountzakis (\cite{K-Riesz15}, Theorem 1) on the sufficiency of multi-tiling property for the existence of exponential Riesz bases for a bounded measurable domain. 
 The  focus of this paper is to extend the results of Theorems A and B  for  infinite locally compact topological abelian (LCA) groups $(G,+)$ with identity $e$ and equipped with a Haar measure.  
 Note that some  LCA groups, like $p$-adic additive groups $\Bbb Q_p$, the field of $p$-adic rational numbers,  which do not posses any lattice. 
Therefore, the classical definition of wavelet sets does not apply to  such groups.  It is also known   that   the   finite abelian group do not posses any wavelet set in the traditional sense (\cite{ILM16}). Therefore, we exclude such  cases here and assume that $G$ is infinite and 
   admits a lattice, i.e., a discrete  subgroup which is co-compact.   For the notion of wavelet sets in $\Bbb Q_p$ and in a non-commutative setting we refer  to 
 \cite{Benedetto-Leon}  and 
  \cite{CM1}, respectively. 
  
  For a characterization of Riesz wavelets  generated by  MRA  see  e.g.,   \cite{Bownik-RieszWave}  and  \cite{HanJan07}.   
  
The purpose of this paper is to establish the necessary formalism to state and prove Theorem A and Theorem B in the setting of locally compact abelian groups. We shall also prove the converse of Theorem B, not only in Euclidean space but also in the general context. 

In the classical setting of $\Bbb R^n$, the definition of a wavelet is associated to two type of   unitary operators: dilations and translations.  The translation operator for the functions on an LCA  group can be defined by the group action analogous with the Euclidean setting. The dilation of a function defined on   $\Bbb R^n$ is given by the action of an invertible matrix $A$ and the associated normalization factor $c=|\det A|^{n/2}$. Therefore the immediate   feeling is that the dilation operator acting on the  functions on a LCA group  must be defined  by using group automorphisms. 
The  only obstacle here would be the normalization factor. Obviously, for discrete groups, the normalization factor is $c=1$ since the Haar measure is the counting measure.  Therefore we shall assume that our  setting is not discrete.     In this case we overcome the problem by using  { modular function} $\Delta$, which is  a  positive  homomorphism     defined on the class of automorphisms of $G$. We will define this function in Section \ref{Preliminaries}.

\section{Statement of main results} 

Borrowing the notion of wavelet collection of sets for the Euclidean setting in (\cite{BB2009}), we have the following definition.

 \begin{definition}\label{collection} Let $G=(G,+)$ be an LCA group with identity $e$ and    the dual group $\hat G$. We equip $G$ by    a Haar measure. 
A collection of measurable subsets $\Omega_i$, $1\leq i\leq m$, of $\hat G$  is called a (Riesz) wavelet  collection of sets for $G$  if there are  countable subset   $A$ of    $\mathcal Aut(G)$ and  countable subsets 
   $\Lambda_i\subset G$  ($1\leq i\leq m$) such that   
$$\cup_{i=1}^m\{\psi_{i,\lambda,a}: \ 1\leq i\leq m, \lambda\in \Lambda_i, a\in A\}$$
is  (Riesz) orthogonal basis for $L^2(G)$, where $\psi_{i,\lambda,a}(x)= \Delta(a)^{1/2} \psi_i(a(x)-\lambda)$ and $\hat\psi_i= 1_{\Omega_i}$. When $m=1$, we may say  $\Omega=\Omega_1$ is  a (Riesz) wavelet set. 
\end{definition}

The study of wavelet
sets  is  an interplay between group theory, geometry, operator theory
and analysis. 
The following result is an extension form of  Theorem  A and  Theorem 2.2 in \cite{BB2009}.  In the sequel, we assume that $\Omega_1, \cdots, \Omega_m$ are $m$ mutual  disjoint and measurable subsets of $\hat G$ with finite and non-zero  Haar measure, and $\Omega:=\cup_{i=1}^m \Omega_i$. 
 \\
 
 \begin{theorem}\label{wavelet set-necessary and sufficient condition}
  Let $\Omega_1, \cdots, \Omega_m$ be  given.  Let 
  $A\subset \mathcal Aut(\hat G)$ and $ \Lambda_i\subset G$, $1\leq i\leq m$, be   countable and non-empty subsets.  Assume that for every $1\leq i\leq m$, the   pair  $(\Omega_i, \Lambda_i)$ is  spectral and  the set $\Omega$ tiles $\hat G$ multiplicatively by the automorphisms $A$.    Then $\{\Omega_1, \cdots, \Omega_m\}$  is a wavelet collection of sets for $L^2(G)$ with respect to the dilations   $A$ and translations $\Lambda_i$. The converse also holds, provided that $e\in\cap_{i=1}^m  \Lambda_i$. 
 \end{theorem}
 
 Theorem \ref{wavelet set-necessary and sufficient condition}  along with a result of Fuglede \cite{Fug74}  yields the following corollary, where  the spectral condition  has been  replaced by additive tiling property.
 
 \begin{corollary}   The set  $\{\Omega_1, \cdots, \Omega_m\}$  is a wavelet collection of sets for $L^2(G)$ with respect to the  translations by lattices $\Lambda_i$ ($1\leq i\leq m$) 
 % with respect to the   translations by lattice  $\Lambda_i$ and dilations $A\subset \mathcal A ut(G)$ 
   if and only if   $\Omega:=\cup_{i=1}^m \Omega_i$ tiles $\hat G$ multiplicatively   and each $\Omega_i$ tiles $\hat G$ additivity with  $\Lambda_i^\perp$. 
 \end{corollary}

 %Note that the translation lattice in the above corollary is $\Lambda^\perp$,  the annihilator of  the lattice $\Lambda$. 
 %
 
 The following result is an extension form of Theorem 
 \ref{wavelet set-necessary and sufficient condition} to the Riesz bases. 
 \begin{theorem}\label{Riesz wavelet set}   
 Let $\Omega_1, \cdots, \Omega_m$  be as above and countable sets $\Lambda_i \subset \hat G$  be given such that   for any $1\leq i\leq m$   the pairs $(\Omega_i, \Lambda_i)$ are Riesz spectral with  Riesz constants $L_i$ and $U_i$.  Moreover, assume that $\Omega$  is a multiplicative tiling for $\hat G$ with respect to 
the automorphisms $A\subset \mathcal Aut(\hat G)$.  
 Then $\Omega_1, \cdots, \Omega_m$ is a 
 Riesz wavelet collection of sets for $L^2(G)$, i.e.,  
   the family  
\begin{align}\label{multi-wavelet set}
\cup_{i=1}^m \{\psi_{i, \lambda, a} : \lambda \in \Lambda_i,  \ a\in A\} 
\end{align}
 is a Riesz basis for $L^2(G)$ where   $\hat\psi_i:= 1_{\Omega_i}$. The   Riesz constants in this case are given by   $L:=\min\{L_i\}$ and $U:=\max\{U_i\}$. 
 
% with unified constants, with respect to the  translation set  $T$ and automorphisms $ A$. 

\end{theorem}

A combination of  Theorem \ref{Riesz wavelet set} and     \cite[Theorem 4.1]{Agoraa-Antezanaa-Cabrelli}  yields  the following  result as an extension version of  Theorem B  for $m\geq 1$.

\begin{corollary}\label{Riesz wavelet set and tiling}    Let $\Omega_1, \cdots, \Omega_m$  be as above.
 Assume that  each 
  $\Omega_i$  multi-tiles $\hat G$ additively by some lattice  $\Gamma_i \leq \hat G$  with multiplicity $k_i \in \Bbb N$, $k_i>1$. Furthermore,  assume that $\Omega=\dot\cup_{i=1}^m \Omega_i$   tiles  $\hat G$ multiplicatively with respect to 
a subset  $A$ of  $\mathcal Aut(\hat G)$. Then  $\{\Omega_1, \cdots, \Omega_m\}$ is a 
Riesz wavelet collection of sets for $G$. 
\end{corollary}

An inverse theorem for the results in Theorem \ref{Riesz wavelet set}   follows. 
 
  \begin{theorem}\label{RB implies disjointness} 
Assume that $\Omega$ is a Riesz wavelet set with respect to the translations by $\Lambda$ and dilations by  $A$.  Then $\Omega$ is a multiplicative tiling for $\hat G$ with respect to the automorphisms $A$ if and only if    $(\hat\alpha(\Omega), \alpha^{-1}(-\Lambda))$  is a Riesz  spectral pair for all $\alpha\in A$. 
\end{theorem}  
  
%{\color{red} Parseval frames... see Benedetto Benedetto\rq{}s paper} \\ 
%$\dot \cup$ designates disjoint union

{\bf Outline of the paper.}  The paper is organized as follows.  After Section \ref{open problems} on some open problems, in 
  Section  \ref{Preliminaries} we recall necessary backgrounds  on the Fourier transform on locally compact abelian group $G$ and wavelet and Riesz bases in $L^2(G)$. 
 Theorem 
 \ref{wavelet set-necessary and sufficient condition} asserts that a collection of measurable  sets with finite and positive measure  in the dual space of a LCA group $G$  generate an orthogonal wavelet bases for $L^2(G)$ if and only if the sets are spectral and admit multiplicative tiling property, provided that  the intersection of all translation sets contains $e$. 
 In section \ref{Wavelet sets} we give the proof of Theorem 
 \ref{wavelet set-necessary and sufficient condition}.  In Section \ref{Reisz wavelet sets} we prove Theorems  \ref{Riesz wavelet set} and  \ref{RB implies disjointness}. These theorems are as an extension version of both necessary and sufficient part of  Theorem \ref{wavelet set-necessary and sufficient condition}. We conclude the paper by presenting some examples of wavelet and Riesz wavelet sets in Section \ref{examples}.

\section{Open problems}\label{open problems} 
%
%Relevant to Theorem  \ref{wavelet set-necessary and sufficient condition} we have the following question:   
 %
%\begin{question}  If $\Lambda_i= \Lambda$, is $\Omega$ a wavelet set? This is equivalent to the following question. 
  %\\
  %If we add $\psi:= \psi_1+ \psi_2$, what happens? See Shannon wavelet 
%\end{question} 
%
 In the example constructed by Wang \cite{Wang2002}, we observed that  the  wavelet set $[-1,-1/2]\cup [1/2,1]$  does    tile the space $\Bbb R$  multiplicatively on $\{\pm 2^n\}_{n\in \Bbb Z}$.   However, the set $\Omega$ has a multiplicative tiling  property with respect to $M$  with multiplicity $2$. In this relation we make the following definition.
 
 \vskip.125in 

\begin{definition} Let $\Omega\subset G$. We say $\Omega$  tiles $G$ multiplicatively  with multiplicity  $k$ if there is a countable subset $A$ of automorphisms of $G$ such that 
 $$\sum_{\alpha\in A} 1_{\alpha(\Omega)} (x) =k \quad a.e. \  x\in G .$$ 
The case $k=1$ is the traditional multiplicative tiling.  
 \end{definition} 
 
 \vskip.125in 
 
\begin{question}\label{multiplicative tiling at level k}
Suppose that $\Omega$ is a wavelet set in a locally compact abelian group $G$ with respect to a dilation set $A\subset \mathcal Aut(G)$ and translation set $\Lambda\subset G$, where $e\not\in \Lambda$. Must $\Omega$    tile $\hat G$ multiplicatively by some $A' \subset \mathcal Aut(G)$, $A\rq{}\neq A$? 
\end{question} 
Notice that, when $e\in \Lambda$, then according to Theorem \ref{wavelet set-necessary and sufficient condition}  the wavelet set is a multiplicative tiling and    $A\rq{}=A$. 
 A   version of  Question \ref{multiplicative tiling at level k} for a non-separable wavelet index is the following. 
 \begin{question}\label{non-separable} Let $\mathcal B:=\{ (\lambda, a): \lambda\in G, \ a\in  \mathcal Auto(G)\}$ be countable. Assume that  for $\psi$ with $\hat\psi=1_\Omega$, the set $\{\psi_{\lambda, a}: (\lambda, a)\in \mathcal B\}$ is an orthogonal basis for $L^2(G)$, i.e., $\Omega$ is a wavelet set with respect to $\mathcal B$. Assume that $(e,a)\in \mathcal B$ for some $a\in G$. 
Must $\Omega$ tile additively and multiplicatively?
 \end{question}

In relation to Corollary  \ref{Riesz wavelet set and tiling} we have the following question. 

\begin{question}  Suppose that 
  $\{\Omega_i\}_{i=1}^m$  is  a Riesz wavelet collection of  sets for $L^2(G)$, which is not orthogonal. Let $\Omega:=\cup_i\Omega$.  Must $\{\alpha(\Omega)\}_{a\in A}$ be a  multiplicative tiling  for $G$ with multiplicity $k\geq 1$? 
 \end{question}

 {\bf Acknowledgements.} The author would like to thank Alex Iosevich and Shahaf Nitzan for helpful conversations.
 Support for this project was partially provided by PSC-CUNY 
   by  PSC-CUNY Award B $\sharp$ 69625-00 47, jointly funded by The Professional Staff Congress and The City University of New York.
 
 \section{Notations and Preliminaries}\label{Preliminaries}
Following \cite{Rudin-Fourier on group}, 
 let $(G, +)$ denote a locally compact abelian group  with the group identity $e$. 
 We assume 
 $G$ is second countable and we equip $G$ with a  normalized Haar measure   $m_G$ which is non-zero and     translation invariant.   Let $\hat G$ denote the dual group of $G$ and $m_{\hat G}$ the Haar measure on $\hat G$.   For any element $\xi$ in the dual group  $\hat G$, we associate the  character $\chi_\xi: G\to \Bbb C$ of $G$ and  denote the duality by $\chi_\xi(x) := \langle \xi, x\rangle$, $x\in G$. When $G$ in $n$-dimensional Euclidean space $\Bbb R^n$, then $\widehat{\Bbb R^n}= \Bbb R^n$ as well and $\langle \xi, x\rangle = e^{2\pi i \xi \cdot x}$,  $\xi,  x\in \Bbb R^n$, is the canonical mapping. 
 
 By the properties of characters,  we have 
 \begin{align}\label{property of character}
 \chi_\xi(-x) = \chi_{-\xi}(x)= \overline{\chi_\xi(x)} . 
 \end{align}
 
 For any $p>0$, analogously to the Euclidean case, we define the space $L^p(G)$ with respect to the measure $m_G$ by 
 \begin{align} 
 L^p(G)= \{f: G\to \Bbb C, \ f \ m_G\text{-measurable and } \  \int_G |f(x)|^p dm_G(x) <\infty\}. 
 \end{align} 
    For any $f\in L^1(G)\cap L^2(G)$, let $\hat f$ denote the  {\it Fourier transform} of $f$ given by 
\begin{align}\label{group-fourier-transform}
\hat f(\xi) = \int_G f(x) \overline{\chi_\xi(x)} dm_G(x)\ , \ \  \xi\in \hat G
\end{align} 
By the Plancherel theorem the definition of the Fourier transform can be extend uniquely to the functions in $L^2(G)$ (See, for example, 1.6.1 in \cite{Rudin-Fourier on group}), so that for any $f, g\in L^2(G)$ the Parseval identity holds: 
 \begin{align}\label{Parseval identity} 
  \int _G f(x) \overline{g(x)} dm_G(x) = \int_{\hat G} \hat f(\xi) \overline{\hat g(\xi)} dm_{\hat G}(\xi)  ~.
  \end{align}
 We denote topological automorphisms of $G$ onto itself by   $ \mathcal Aut(G)$.    
 For $\alpha\in \mathcal Aut(G)$,  the  {\it adjoint homomorphism}  $\hat \alpha: \hat G\to \hat G$  is given as following: For any character $\chi_\xi$, $\hat \alpha(\chi_\xi)$ is a character and its duality is given by 
$$\hat \alpha(\chi_\xi)(x) := \chi_\xi(\alpha(x))= \langle \xi, \alpha(x)\rangle \quad x\in G.$$
 It is  easy to show that $\hat \alpha$ is an automorphism (one-to-one  homomorphism) of $\hat G$ onto $\hat G$ and  its inverse $\hat \alpha^{-1}$  acts  on $\hat G$ by 
$$\hat \alpha^{-1}(\chi_\xi)(x):= \chi_\xi(\alpha^{-1}(x))= \langle \xi, \alpha^{-1}(x)\rangle .$$ 
This implies that  $\hat \alpha^{-1} = \widehat{\alpha^{-1}}$. 
By the definition of the adjoint homomorphism, the following result  immediately follows. 
\begin{lemma}\label{property of adjoint}   
Given   any automorphism $\alpha: G\to G$,  $\xi\in \hat G$ and   $x\in G$  
$$\hat\alpha(\chi_\xi)(x) = \chi_{\hat\alpha(\xi)}(x)  \quad  {\text and} \quad \hat\alpha^{-1}(\chi_\xi)(x)=  \chi_{\hat\alpha^{-1}(\xi)}(x)$$

%$$\hat\alpha(\chi_\xi)(x) = \chi_\xi(\alpha(x)) = \langle \hat\alpha(\xi), x\rangle  = \chi_{\hat\alpha(\xi)}(x) .$$ 
\end{lemma}
A subgroup $\Lambda$ of $G$ is called a {\it lattice} if it is discrete and co-compact, i.e.,  the quotient group $G/\Lambda$ is compact. $G$ is second countable, therefore any lattice in $G$ is  countable (\cite[Section 12, Example 17]{Pontryagin}).
  The {\it annihilator}  or {\it dual lattice} of a lattice $\Lambda$ is given by 
  $$\Lambda^\perp=\{ \xi \in \hat G: \  \chi_\xi(\lambda)=1 \quad  \forall \lambda\in \Lambda\}.$$
  Then $\Lambda^\perp$ is also a lattice in $\hat G$. By the duality theorem \cite[Lemma 2.1.3]{Rudin-Fourier on group}, $\Lambda^\perp$ is topologically isomorphic to the dual of $G/\Lambda$, that is, $\Lambda^\perp \cong \widehat{(G/\Lambda)}$.

  Let  $\alpha: G\to G$ be an automorphism. Then 
  $$ \hat \alpha (\Lambda^\perp) = (\alpha^{-1}(\Lambda))^\perp .$$ 
 Given a countable set  $\Lambda$ in $G$ and a non-zero measurable set $\Omega\subset G$, we say $\Omega$ {\it tiles} $G$  additively by  $\Lambda$ with multiplicity  $k\in \Bbb N$ if 
 \begin{align} 
 \sum_{\lambda\in \Lambda} 1_\Omega(x-\lambda) = k \quad a.e. \ x\in G ~ . 
 \end{align} 
 If $k=1$, then we simply say $\Omega$ {\it tiles}    $G$  additively  by $\Lambda$.

  Given   $\Omega\subset \hat G$ with non-zero measure and a countable set   $T\subset G$, we say $(\Omega, T)$ is a {\it  spectral pair} if the countable  family  $\mathcal E_T:=\{ \chi_t: \ t\in T\}$
is an orthogonal basis for $L^2(\Omega)$. The pair is called  
  {\it Riesz spectral pair} if  $\mathcal E_T$ 
constitute a Riesz basis for $L^2(\Omega)$. 
The set $\Omega\subset G$ is   {\it multiplicative tiling  }   if there is a set of automorphisms  $\mathcal A$ of $G$ such that   

$$\sum_{\alpha\in \mathcal A} 1_{\alpha(\Omega)}(x) = 1 \quad a.e.  \ x\in \Bbb G.$$
 
 For general definition of multiplicative tiling sets in a measure space $(X, \mu)$ see \cite{Olafsson_wavelet sets}.

{\it Modular function.} Associated to a locally compact abelian group $G$ with  the  Haar measure $m_G$, the {\it modular function} $\Delta:  \mathcal Aut(G)\to (0,\infty)$ is a continuous homomorphism   such that for 
any measurable set $\Omega\subset G$ there holds 
\begin{align}\label{measure of a dilated set} 
m_G(\alpha(\Omega)) = \Delta(\alpha)m_G(\Omega).
\end{align}  
It is well-known that the  modular function $\Delta$ exists. For this, consult  \cite[26.21]{HwRoI}. (The existence of such map for a larger class of continuous group homomorphisms of $G$ onto $G$ has also been proved in \cite[Theorem 6.2]{Bownik-Ross}.) 
The relation (\ref{measure of a dilated set}) implies that 
for any  integrable function $f$ on $G$  with respect to the Haar measure $m_G$  the equation 
\begin{align}\label{modular-function}
\int_G f(\alpha(x)) dm_G(x) = \Delta(\alpha)^{-1} \int_G f(x) dm_G(x) 
\end{align}
holds  for all  automorphism $\alpha$ of $G$.  For our purpose, we  shall assume that $\Delta$ {\it is not}  identical to one. 

{\it Dilations and translations.} Let $G$ be a LCA group. 
There are two unitary operators associated to the definition of wavelets on the group: translation and dilation. For $x\in G$, the {\it translation operator} $\tau_x$ is defined on function $f: G\to \Bbb C$ and  given by 
   $\tau_x f(y) = f(y-x)$.  By the translation invariant property of Haar measure,   $\tau_x$ is unitary restricted to the functions in $L^2(G)$.  For a fixed  automorphism $\alpha\in \mathcal Aut(G)$, we define  $\delta_a$  on $f$ by 
    $$\delta_\alpha f(x) = \Delta(\alpha)^{1/2} f(\alpha(x)) , \ \forall \ x\in G .$$
    When $\Delta(\alpha)\neq 1$, we call $\delta_\alpha$ {\it dilation operator}. In the sequel we shall assume that the operator $\delta_\alpha$ is a dilation.  
It is easy to see that by    (\ref{modular-function})  the operator $\delta_\alpha$ is an unitary operator. 
  For a mapping $f: G\to \Bbb C$,  $\lambda\in \Lambda$ and $\alpha\in \mathcal Aut(G)$  we denote by $f_{\alpha,\lambda}$  the ``dilation and translation copy\rq\rq{} of $f$ with respect to $\alpha$ and $\lambda$, respectively, and define  it by 
 $$f_{\alpha,\lambda}(x):= \delta_\alpha \tau_\lambda f (x)= \Delta(\alpha)^{1/2} f(\alpha(x) -\lambda) . $$ 
 For $x\in G$ and $g: \hat G\to \Bbb C$,  the {\it modulation operator} $M_x$ on $g$ is   given by 
 $$M_xg(\xi) = \langle \xi, x\rangle g(\xi)  \quad \xi\in \hat G .$$ 
The following result is immediate by the group Fourier transform.

 \begin{lemma}\label{Fourier on wavelet}  Let $\alpha$ be an automorphism of $G$  and  $\lambda\in G$. Then for any $f\in L^1(G)\cap L^2(G)$
 $$    \widehat f_{\alpha,\lambda}(\xi) =  \widehat{\delta_\alpha \tau_\lambda   f }(\xi)= 
 \Delta(\alpha)^{-1/2}    (M_{-\lambda} 
        \hat f)(\hat \alpha^{-1}(\xi)),  \quad \xi\in \hat G .$$ 
 %where the equation holds pointwise for $\xi\in \hat G$. 
 \end{lemma}

\begin{proof}  The result obtains from the application of the Fourier transform on the  dilation and translation operators, as follows.  Let $\xi\in \hat G$. Then by  (\ref{group-fourier-transform}), the following holds. 
\begin{align} \widehat f_{\alpha,\lambda}(\xi) &= \int_{G} f_{\alpha, \lambda}(x) \overline{\chi_\xi(x)} dm_G(x)\\\notag
  &= \Delta(\alpha)^{1/2} \int_{G}   f(\alpha(x)-\lambda) \overline{\chi_\xi(x)} dm_G(x)  \\\notag
 &= \Delta(\alpha)^{-1/2}\int_{G}  f(x) \overline{\chi_\xi(\alpha^{-1}(x+\lambda))}  dm_G(x)  \\\notag
  &= \Delta(\alpha)^{-1/2}  \overline{\chi_\xi(\alpha^{-1}(\lambda))} \int_{G}  f(x) \overline{\chi_\xi(\alpha^{-1}x)}  dm_G(x) \\\notag
   &= \Delta(\alpha)^{-1/2} \overline{\chi_\xi(\alpha^{-1}(\lambda))} \int_{G} f(x) \overline{\chi_{\hat\alpha^{-1}(\xi)} (x)}  dm_G(x) \\\notag
   &=  \Delta(\alpha)^{-1/2}  \overline{\chi_\xi(\alpha^{-1}(\lambda))} \hat f(\hat \alpha^{-1}(\xi))  \\\notag
    %&=  \Delta(\alpha)^{-1/2}     e_{\lambda^{-1}}(\hat\alpha^{-1}\xi) \hat f(\hat\alpha^{-1}\xi) \\\notag
      &=  \Delta(\alpha)^{-1/2}    (M_{-\lambda} 
        \hat f)(\hat \alpha^{-1}(\xi)). 
\end{align} 
This completes the proof. 
Note that within the above lines we used the property  
 $\chi_\xi(\alpha^{-1}(x)) =  \chi_{\hat\alpha^{-1}(\xi)}(x)$ proved   in  Lemma \ref{property of adjoint} and the equation $\overline{\chi_\xi(x)}= \chi_\xi(-x)$ by (\ref{property of character}).  
\end{proof}

 Assume that   $f\in L^2(G)$ with   $\hat f= 1_\Omega$.    As a result  of the previous lemma   
  one has 
 \begin{align}\label{Fourier-transform-of-wavelets}
   \widehat f_{\alpha,\lambda}(\xi) =    \Delta(\alpha)^{-1/2}     \chi_{\hat\alpha^{-1}(\xi)}(-\lambda) 1_{\hat\alpha(\Omega)}(\xi)  \quad \xi\in\hat G.
 \end{align} 
   {\it Convention.} In the sequel, we shall use the notation ``exponential\rq\rq{}  $e_x(\xi):= \chi_\xi(x)$ for   $x\in G$ and $\xi\in \hat G$.

{\it Riesz bases.} A  countable sequence $\{x_n\}_{n\in I}$ in a Hilbert space $\mathcal H$ is a Riesz basis for $\mathcal H$ if the sequence is of the form $\{Ue_n\}_{n\in I}$  for some orthonormal basis $\{e_n\}_{n\in I}$ for $\mathcal H$ and an invertible and bounded linear  operator $U: \mathcal H\to \mathcal H$. This definition is equivalent to say that the sequence $\{x_n\}_{n\in I}$ is dense in $\mathcal H$ and there is  finite and positive constants $A, B$ such that for any finite sequence $\{c_n\}_{n\in I}$ we have 
\begin{align}\label{RB-Ineq}
A\sum_{n\in I}|c_n|^2 \leq \left\| \sum_{n\in I} c_n x_n\right\|_{\mathcal H}^2 \leq B\sum_{n\in I} |c_n|^2 .
\end{align}
It is obvious that if the upper and lower estimations in  (\ref{RB-Ineq}) hold for any  finite sequence $\{c_n\}_{n\in I}$,  then they also hold for any sequence $\{c_n\}_{n\in I}\in \ell^2(I)$. For more about Riesz bases for any Hilbert space we refer to  \cite{Ole_book_2008}. 
 
  The following result is well-known, 
  but we give a proof below to keep the presentation as self-contained as possible. 
 
 \begin{lemma}\label{Riesz basis is onb} Let $\{x_n\}_{n\in I}$ be a Riesz basis for a Hilbert space $\mathcal H$. Assume that for any finite sequence $\{c_n\}$ the identity  holds:
 \begin{align}\label{Parseval-type identity}  
 \| \sum_n c_n x_n\| = (\sum_n |c_n|^2)^{1/2}
 \end{align}
 Then $\{x_n\}$ is an orthonormal basis. 
 \end{lemma} 
 
 \begin{proof} 
 We only need to show that for any $m, n\in I$, $\langle x_n, x_m\rangle = \delta_{m,n}$ %To this end,  fix $M$ and put $c_M:=1$  and $c_n:=0$ for all $n\neq M$. Then 
The identity (\ref{Parseval-type identity}) implies that $\|x_M\| =1$. 
 For the orthogonality of the vectors we proceed as follows: By (\ref{Parseval-type identity}),  it is immediate that  $\|x_n +x_m\|^2= 2$ for any  $n\neq m$. From the other hand    
 $\|x_n + x_m\|^2= 2+ 2{\text Re} \langle x_n, x_m\rangle .$
 This implies that we must have ${\text Re} \langle x_n, x_m\rangle =0 $. By replacing  $x_n$ by $ix_n$, we also  obtain ${\text Im} \langle x_n, x_m\rangle = 0$. These imply that $x_n$ and $x_m$ are orthogonal, and we are done. 
 \end{proof}

\begin{definition} Let $H$ be a infinite-dimensional Hilbert spaces. An infinite collection $\{x_n\}$ of vectors in $H$ is $w$-linearly independent if for any sequence $\{c_n\}$ such that $\sum_n c_n x_n$ converges to zero in the norm of $H$  is identically zero. 
\end{definition} 
It is well-known that any $w$-linearly independent sequence  is linearly independent. However, the converse does  not always hold. 

\begin{lemma}\label{Riesz basis is a frame} A sequence $\{x_n\}$ in a Hilbert spaces is   a Riesz basis for $H$ if and only if $\{x_n\}$ is $w$-linearly independent and there are constants $0<A\leq B< \infty$ such that for any vector $x\in H$, 
 \begin{align}\label{frame inequality}
 A\|x\|^2\leq \sum_n |\langle x,x_n\rangle|^2 \leq B \|x\|^2 . 
 \end{align}
\end{lemma} 

The inequality in the preceding lemma is called the {\it frame inequality}. 
 The $w$-linearly independency of Riesz basis  follows from the fact that any Riesz basis has a biorthogonal sequence. For the proof we refer to \cite{Ole_book_2008}. 

 \section{Proof of Theorem \ref{wavelet set-necessary and sufficient condition}}\label{Wavelet sets}
 
 For the rest of this paper we fix the following notations unless it is stated otherwise: We let   $A$  denote a countable subset of   $\mathcal Aut(G)$, the group of automorphisms of $G$ onto $G$ with $\Delta(\alpha)\not\equiv 1$, $\alpha\in \mathcal Aut(G)$,   $\Lambda$  be a countable subset in $G$, and $\Omega$ be a subset of $\hat G$ with finite and non-zero     Haar measure. 
  %For any $x\in G$, we define the ``exponential\rq\rq{} $e_x: \hat G\to \Bbb C$ by $e_x(\xi):= \langle \xi, x\rangle$. 
  
The first result of this section shows that the spectral property of a set    is preserved under the action of  automorphisms. 
\begin{lemma}\label{spectral-pair} Assume that $(\Omega, \Lambda)$ is a spectral pair and    $\alpha\in A$. Then  the pair $(\hat \alpha(\Omega), \alpha^{-1}(\Lambda^{-1}))$ is spectral, where $\Lambda^{-1}= \{-\lambda: \  \lambda\in \Lambda\}$.  
\end{lemma} 

\begin{proof} Define   $T_\alpha: L^2(\Omega) \to L^2(\hat\alpha(\Omega))$ by $T_\alpha(g)(\xi)=   \Delta(\hat\alpha)^{-1/2} \overline{g(\hat\alpha^{-1}(\xi))}, \ \xi\in \hat \alpha(\Omega)$.  $T$ is linear and 
 
 $$T_\alpha(e_\lambda)= \Delta(\hat\alpha)^{-1/2} e_{\alpha^{-1}(-\lambda)} \ , \quad       \lambda\in \Lambda$$
 
$T$ is a   unitary isomorphism  and $(\Omega, \Lambda)$ is a spectral pair, then $\{\Delta(\hat\alpha)^{-1/2} e_{\alpha^{-1}(-\lambda)}\}_{\lambda\in \Lambda}$ is an orthogonal basis for $ L^2(\hat\alpha(\Omega))$. This completes the proof of the lemma. 
\end{proof}

 \begin{proof}[ Proof of Theorem  \ref{wavelet set-necessary and sufficient condition}]   
 We shall ``somehow"   adjust   the proof of Theorem A for our situation.  Fix $A\subset \mathcal Aut(G)$ and $\Lambda_i\subset G$ and 
 let $\{\Omega_i\}_{i=1}^m$ be a collection of disjoint subsets in $\hat G$ such that $(\Omega_i, \Lambda_i)$ are spectral pairs. 
 
 Let  $\varphi_i\in L^2(G)$ be the inverse Fourier transform of $1_{\Omega_i}$, i.e.,    $\hat \varphi_i = 1_{\Omega_i}$. Define 
 \begin{align}\label{wavelet family}
   \mathcal W:= \cup_{i=1}^m\{\varphi_{i, \alpha, \lambda}: \alpha\in A, \lambda\in \Lambda_i\} .
   \end{align} 
   Then  by an  application of the Fourier transform and  the equation  (\ref{Fourier-transform-of-wavelets}), $\mathcal W$ is an orthogonal basis for $L^2(G)$ if and only if    the system $ \widehat{\mathcal W}$      is an orthogonal  basis for $L^2(\hat G)$: 
      \begin{align}\label{Fourier wavelet family}
  \widehat{\mathcal W}&=  \cup_{i=1}^m\{\widehat{\varphi_{i, \alpha, \lambda}}: \alpha\in A, \lambda\in \Lambda_i\}\\\notag
 & = \cup_{i=1}^m \{\Delta(\alpha)^{-1/2}  e_{\alpha^{-1}(-\lambda)} 1_{\hat\alpha(\Omega_i)}:  \alpha\in A, \lambda\in \Lambda_i \} .
  \end{align}

By the assumption, 
each pair $(\Omega_i, \Lambda_i)$ is a spectral. Therefore, for  any  $\alpha\in A$,  by Lemma \ref{spectral-pair}   the collection $\{e_{\alpha^{-1}(-\lambda)} 1_{\hat \alpha(\Omega_i)}:   \lambda\in \Lambda_i\}$ is an orthogonal basis for $L^2(\hat \alpha(\Omega_i))$. 
From the other hand, 
by multiplicative tiling property of $\{\Omega_i\}_{i=1}^m$,  for any   $\alpha, \beta \in A$ and $1\leq i, j \leq m$ we have 
 \begin{align}\label{intersection of sets} 
 m_{\hat G}(\hat\alpha(\Omega_i) \cap \hat\beta(\Omega_j)) =m_{\hat G}(\hat\alpha(\Omega_i)) \delta_{\alpha, \beta} \delta_{i, j}.    
 \end{align}
  Thus the elements of $\widehat{\mathcal W}$ are mutual orthogonal.  
%\footnote{ Note that $\hat\alpha(\Omega_i) \cap \hat\beta(\Omega_j)\subset \hat\alpha(\Omega) \cap \hat\beta(\Omega)$  where 
%$\Omega= \cup_i \Omega$.}
\iffalse 
The relation (\ref{intersection of sets}) can be verified as follows: 
  
  For the case $\alpha=\beta$ and $i=j$ there is nothing to prove. Let $\alpha\neq \beta$ and  $i, j$ be arbitrary.  Take $M:= \hat\alpha(\Omega_i) \cap \hat\beta(\Omega_j)$. 
    Assume that $m_{\hat G}(M)>0$. Therefore $M$ is a subset of $\hat\alpha (\Omega) \cap \hat \beta (\Omega)$ with positive measure. This contradicts the assumption that $\Omega$ is a multiplicative tiling for $\hat G$ with respect to $A$. 
    %
    Now assume $\alpha=\beta$ and $i\neq j$. Assume  $S$ is a subset of $  \hat\alpha(\Omega_i) \cap \hat \alpha(\Omega_j)$  such that $m_{\hat G}(S)>0$.         then $(\hat\alpha)^{-1}(S)$ is a subset of $ \Omega_i \cap \Omega_j$ with measure 
    $m_{\hat G}((\hat\alpha)^{-1}(S)) = \Delta(\hat\alpha)^{-1}m_{\hat G}(S)$. 
    Since $m_{\hat G}(S)>0$ and the modular function is positive, then we must have  $m_{\hat G}((\hat\alpha)^{-1}(S)) >0$. This contradicts the mutual disjointness assumption of the sets $\Omega_i$. This completes the proof of our assertion. 
    \fi
  %
  The completeness of the system $\widehat{\mathcal W}$ in $L^2(\hat G)$ holds by the  $A$-multiplicative tiling of $\hat G$ and   the decomposition  
  \begin{align}\label{decomposition}
  L^2(\hat G)= \bigoplus_{\alpha\in A} \bigoplus_{1\leq i\leq m}L^2(\hat \alpha(\Omega_i)). 
  \end{align}
  Conversely, let $\{\Omega_i\}_{i=1}^m$  be a wavelet collection of sets with respect to the automorphisms $A$ and  translations $\Lambda_i$, $1\leq i\leq m$. Assume that $e\in \Lambda_i$ for all $i$. 
  Then the  elements of the  collection $\{\varphi_{i,\alpha}:= \varphi_{i,\alpha,0}: \  \alpha\in A, \ 1\leq i\leq m\}$ are mutual  orthogonal   and for   any     $(i,\alpha_1)\neq (j,\alpha_2)$     we have 
 \begin{align} 
0&= \int_{G}  \varphi_{i,\alpha_1}(x)   \overline{\varphi_{j,\alpha_2}(x)}  dm_G(x)\\\notag 
 & = \int_{\hat G} \hat \varphi_{i,\alpha_1}(\xi)\overline{\hat \varphi_{j, \alpha_2}(\xi)}  dm_{\hat G}(\xi)\\\notag
&= \int_{\hat G} 1_{\hat\alpha_1(\Omega_i)}(\xi) 1_{\hat\alpha_2(\Omega_j)}(\xi) dm_{\hat G}(\xi) \\\notag
&= m_{\hat G}(\hat\alpha_1(\Omega_i) \cap \hat\alpha_2(\Omega_j)) . 
 \end{align} 
 
 This implies that for 
 $\Omega=\dot \cup_{i=1}^m \Omega_i$, 
 the subsets $\hat\alpha(\Omega),\  \alpha\in A$ are  mutual disjoint  and  form a   covering for $\hat G$. Indeed,   assume that there is a subset  $W\subset \hat G$ of positive and finite Haar measure such that   $|W\cap \hat\alpha(\Omega)|=0$ for all $\alpha\in A$. Then  $|W\cap \hat\alpha(\Omega_i)|=0$  for all $1\leq i\leq m$. Then for the function  $f:G\to \Bbb C$ with  $\hat f = 1_W$ is  $\langle f, \varphi_{i,\alpha, \lambda}\rangle  = 
 \langle 1_W, \widehat{\varphi_{i, \alpha, \lambda}}\rangle=0$ for all 
  $i, \alpha$ and
   $\lambda\in \Lambda_i$. From the other hand, 
$\{\Omega_i\}_{i=1}^m$  is a wavelet collection of sets. Then $f$ must be identical to zero. This contradicts the assumption that $|W|>0$, hence $\Omega$ tiles $\hat G$ multiplicatively by $A$. 
 
For the rest, we prove that  the pairs $(\Omega_i, \Lambda_i)$ are spectral. To this end,  according to Lemma \ref{spectral-pair}, it is sufficient to show that for any $\alpha\in A$ each pair $(\hat\alpha(\Omega_i), \alpha^{-1}(-\Lambda_i))$ is a spectral. Or, equivalently, the system 
$$\{\Delta(\alpha)^{-1/2} e_{\alpha^{-1}(-\lambda)} 1_{\hat\alpha(\Omega_i)}: \ \lambda\in \Lambda_i \}$$ 
is orthogonal and complete in $L^2(\hat\alpha(\Omega_i))$. The orthogonality holds from the assumption that $\{\Omega_i\}_{i=1}^m$ is a wavelet collection of sets. For the completeness, assume that $g\in L^2(\hat\alpha(\Omega_i))$ such that $\langle g, e_{\alpha^{-1}(-\lambda)}\rangle_{L^2(\hat\alpha(\Omega_i))} =0$ for all $\lambda\in \Lambda_i$. This, along  the mutual disjointness of $\{\hat\alpha(\Omega_j): \alpha\in A, j\}$, yields 
$$\|g\|^2 =   \sum_{\beta\in A} \sum_{j=1}^m   \sum_{\lambda\in \Lambda_j}   |\langle g 1_{\hat\alpha(\Omega_i)}, \Delta(\beta)^{-1/2} 1_{\hat\beta(\Omega_j)} e_{\beta^{-1}(-\lambda)}\rangle|^2= 0 ,$$ 
which implies   that $g$ must be zero. Thus 
 $(\hat \alpha(\Omega_i), \alpha^{-1}(\Lambda^{-1}))$ is a spectral pair and we are done. 
 \end{proof} 
 
\begin{remark} Note that the multiplicative tiling property of the set $\Omega$ for $\hat G$  is required for the completeness of the wavelet system in $L^2(G)$. However, the results hold for a  subspace of $L^2(G)$ if we only assume that   $\Omega$ is a multiplicative tiling for a subset of $\hat G$. \\
\end{remark}

 \begin{example}\label{epsilon example} We shall use the multiplicative tiling property of the 
 Shannon set to construct a simple example of one dimensional orthogonal wavelet basis for a subspace of $L^2(\Bbb R)$.   Let $\Omega_1:=[-1,-1/2-\epsilon_1)$ and $\Omega_2:= (1/2+\epsilon_2, 1]$ where $0\leq \epsilon_i<1/2$ . Then $(\Omega_i,   \alpha_i \Bbb Z)$  is a spectral pair  with   $\alpha_i= 2(1-2\epsilon_i)^{-1}$.  
 The set $\Omega= \Omega_1 \cup \Omega_2$ is 
 a multiplicative tiling for some subset $U\subset \Bbb R$  with respect to the dyadic dilations $\{2^n: \ n\in \Bbb Z\}$. This  along with  the result of Theorem \ref{wavelet set-necessary and sufficient condition} implies that  the  collection $\widehat{\mathcal W}$ (\ref{Fourier wavelet family}) is an orthogonal basis for  $L^2(U)$. 
 \end{example}

 \section{Proof of Theorems  \ref{Riesz wavelet set} and  \ref{RB implies disjointness}}\label{Reisz wavelet sets}
  To prove Theorem \ref{Riesz wavelet set}, we need to prove 
 an analogous result to    Lemma  \ref{spectral-pair} for Riesz spectral pairs, as follows. 
\begin{lemma}\label{riesz spectral pair} 
 Let $(\Omega, \Lambda)$  be a Riesz spectra with Riesz constants $A$ and $B$. Then for any $\alpha\in A$, the pair $(\hat \alpha(\Omega), \alpha^{-1}(\Lambda^{-1}))$ is a Riesz spectral with  unified Riesz constants. 
 \end{lemma} 
 \begin{proof}  The proof is straight-forward and uses an application of the operator $T_\alpha$ in Lemma \ref{spectral-pair}, and the fact the 
 image of a Riesz basis under any unitary map is Riesz basis with the unified constants.  
 \end{proof} 
\begin{proof}[Proof of Theorem \ref{Riesz wavelet set}] 
    By unitary Fourier transform 
it is sufficient to show that   
$$\widehat{\mathcal W}:=\cup_{i=1}^m \{\hat\psi_{i, \alpha, \lambda_i}\} = \cup_{i=1}^m \{ \Delta(\alpha)^{-1/2}   e_{\alpha^{-1}(-\lambda_i)}(\xi)  1_{\hat\alpha(\Omega_i)}: \alpha\in A,  \lambda_i\in \Lambda_i\}$$ 
 is  a Riesz basis for  $L^2(\hat G)$. By the assumptions, for any $i$, the pair 
 $(\Omega_i, \Lambda_i)$ is  Riesz spectral.  Therefore  by Lemma \ref{riesz spectral pair}   for any $\alpha\in A$ 
   the pair $(\hat \alpha(\Omega_i), \alpha^{-1}(-\Lambda_i))$ is  Riesz spectral    with the unified constants. 
 Then the completeness of the system $\widehat{\mathcal W_{A,T}}$ in $L^2(\hat G)$ holds by the $A$ multiplicative tiling assumption of   $\Omega=\cup_i \Omega_i$ for $\hat G$ and the  decomposition  $L^2(\hat G) =  \bigoplus_{\alpha\in A} L^2(\hat\alpha(\Omega)) =  \bigoplus_{\alpha\in A, 1\leq i \leq m} L^2(\hat\alpha(\Omega_i))$. To prove the Riesz sequence approximation inequalities, we continue as follows: 
  Let $\{c_{\alpha,\lambda_i}: \alpha\in A, \lambda_i\in \Lambda_i, i=1, \cdots, m\}$ be a finite collection of numbers. Then, due the disjointness of $\alpha(\Omega_i)$, $i$,   and that the pairs $(\alpha(\Omega_i), \Lambda_i)$ are Riesz spectral with upper bound $U_i$,   we can write the following: 
\begin{align}\|\sum_{\alpha,i,\lambda_i} c_{\alpha,\lambda_i} \hat\psi_{i,\alpha, \lambda_i}\|_{L^2(\hat G)}^2 
 &= 
 \sum_\alpha \sum_{i=1}^m \|\sum_{\lambda_i} c_{\alpha,\lambda_i} \hat\psi_{i,\alpha, \lambda}\|_{L^2(\hat\alpha(\Omega_i
 ))}^2 \\
 &\leq \sum_{i=1}^m U_i  \sum_\alpha \sum_{\lambda_i} |c_{\alpha,\lambda_i}|^2  \\
 &=  U  \sum_{\alpha,i, \lambda_i} |c_{\alpha,\lambda_i}|^2 \ , 
 \end{align} 
 with $U=\max\{U_i: \ 1\leq i\leq m\}$. 
The lower estimate for Riesz sequence can be obtained in the same fashion. 
 This completes the proof of the theorem.  
  \end{proof} 
 
Note that the completeness of the Riesz system in the previous theorem can also be  obtained  by a different approach:  Let $f\in L^2(G)$ such that $\langle f, \hat\psi_{i,\alpha, \lambda}\rangle =0$  for all $\lambda, \alpha$ and $1\leq i\leq m$. Fix $\alpha$ and $i$.  Then $\langle f1_{\alpha(\Omega_i)}, \hat\psi_{i,\alpha, \lambda}\rangle =0$ for all $\lambda\in \Lambda_i$. Since $(\Omega_i, \Lambda_i)$ is a Riesz spectral, and $\{ \hat\psi_{i,\alpha, \lambda}\}$ is complete in $L^2(\alpha(\Omega_i))$, then we must have $f1_{\alpha(\Omega_i)}=0$. But $f=\oplus_{i,\alpha} f1_{\alpha(\Omega_i)}$, which implies  that  $f$ must be  zero.

  \begin{proof}[Proof of Corollary \ref{Riesz wavelet set and tiling}]   
  Assume that $\Omega_i$  tiles $\hat G$ additively  by some lattice $\Gamma_i$ with multiplicity $k_i$. Then by \cite[Theorem 4.1]{Agoraa-Antezanaa-Cabrelli}, for each $i$,  there is a countable set $\Lambda_i\subset G$ such that $(\Omega_i, \Lambda_i)$ is a Riesz spectral pair in $\hat G$. The rest of the proof is immediate from Theorem \ref{Riesz wavelet set}.  
 \end{proof}
 
 \begin{example} Let $\Lambda$ be a lattice such that, $\Omega_1, \Omega_2, \cdots, \Omega_k$, be  $k$ mutual disjoint  fundamental domains for $\Lambda$. Then $\Omega= \cup_{i=1}^k \Omega_i$ is an additive tiling set for $\hat G$ with multiplicity  $k$ with respect to   translations $\Lambda$. Moreover, assume that $\Omega$ is a multiplicative tiling with respect to  a subset of automorphisms.  Then   $\Omega$ is a Riesz wavelet set.  In this connection, see 
   Example \ref{epsilon example} in Section \ref{Wavelet sets}. \\
  \end{example}

  \begin{remark} 
   Let $(\Omega_i, \Lambda)$, $1\leq i\leq m$,  be spectral pair, i.e., the exponentials $\mathcal E_\Lambda$ form an orthogonal basis for $L^2(\Omega_i)$. If   $\{\Omega_i\}$ are mutual disjoint, then   $\Omega= \cup_{i=1}^m \Omega_i$ is a Riesz spectral set with spectra $\Lambda$. 
 To see this,  
note that for any finite $\{c_\lambda\}_{\lambda\in \Lambda}$ we have 
\begin{align}\label{equation}
\|\sum_{\lambda} c_\lambda e_\lambda\|_{L^2(\Omega)}^2 = \sum_{1\leq i\leq m} \|
 \sum_{\lambda} c_\lambda e_\lambda\|_{L^2(\Omega_i)}^2  = \sum_{1\leq i\leq m}  
 \sum_{\lambda} |c_\lambda|^2.  
 \end{align} 
 
 The completeness of  $\mathcal E_\Lambda$ in $L^2(\Omega_i)$  for    disjoint sets $\Omega_i$, $1\leq i\leq m$,  implies that  $\mathcal E_\Lambda$ is complete in $L^2(\Omega)$, too.     This  along the equation (\ref{equation}) and Lemma \ref{Parseval-type identity}  completes the proof. 
 \end{remark}

\begin{proof}[Proof of Theorem \ref{RB implies disjointness}] To prove the ``if\rq\rq{} part,  let $\alpha\neq \beta$ and let  $M$ be any subset   of  $\hat\alpha(\Omega)\cap \hat\beta(\Omega)$.  Put $f:=1_M$.  We  show that  $M$ must have zero measure. Since $f\in L^2(\hat G)$ and also belongs to the spaces   $ L^2(\hat\alpha(\Omega)) $ and $ L^2(\hat\beta(\Omega))$, then by the  theorem\rq{}s assumptions, there are $l^2$ sequences  $\{c_{\theta, \lambda}: \ \theta\in  A, \lambda\in \Lambda\}$,   $\{d_{\lambda}^\alpha: \ \lambda \in \Lambda\}$ and $\{d_{\lambda}^\beta: \ \lambda \in \Lambda\}$
 for which  we can write 
$$1_M= \sum_{\theta, \lambda} c_{\theta, \lambda} \widehat{\delta_\theta \tau_\lambda \psi} = \sum_{\lambda } d_{\lambda}^\alpha \widehat{\delta_\alpha \tau_{\lambda} \psi} =  \sum_{\lambda} d_{\lambda}^\beta \widehat{\delta_\beta \tau_{\lambda} \psi} . $$
Since the Riesz bases are $w$-linearly independent (Lemma \ref{Riesz basis is a frame}),  the preceding first and second  equalities imply that we must have $ c_{\alpha, \lambda}= d_\lambda^\alpha$ and $ c_{\theta, \lambda}= 0$ for $\theta\neq \alpha$. With a similar argument, by the first and third  equalities  we obtain $c_{\beta, \lambda}= d_\lambda^\beta$ and 
$ c_{\theta, \lambda}= 0$ for $\theta\neq \beta$. These  conclude that $d_\lambda^\alpha = d_\lambda^\beta=0$ for all $\lambda$, hence $1_M$ is the zero function and $|M|=0$. This  implies the measure disjointness of $\alpha(\Omega)$. To prove that $\{\hat \alpha(\Omega): \alpha\in A\}$ is a cover for $\hat G$, we shall use a  
  contradiction approach and apply the frame inequality (\ref{frame inequality}) in Lemma \ref{Riesz basis is a frame} for the wavelet Riesz basis $\mathcal W=\{\delta_\alpha \tau_\lambda \psi: \alpha\in A, \ \lambda\in \Lambda\}$, respectively. 

To prove the ``only if\rq\rq{} part, notice that   the disjointness of   sets $\hat\alpha(\Omega)$ with the frame inequality (\ref{frame inequality}) imply  that for any given $\alpha\in A$, the exponentials $\mathcal E_\Lambda$ are complete in $L^2(\hat\alpha(\Omega))$. The Riesz inequality for $\mathcal E_\Lambda$  is a direct implication of the Riesz inequality for the wavelet system $\mathcal W$. 
\end{proof}

{\it Remark:} Note that  in  Theorem \ref{RB implies disjointness}, only the 
$w$-linearly independency of the  system  $\{\delta_\alpha \tau_t \psi\}$ yields 
  the disjointness of  sets $\alpha(\Omega)$.

 \section{Examples}\label{examples}
 \begin{example}\label{wavelet sets}
 Let $\Omega= [-1,1]^d \setminus [-1/2,1/2]^d$. Fix a lattice $\Lambda$ and  write $\Omega$   as a  finite union  of $\Omega_i$ such that  each $\Omega_i$ is a  tiling with respect  the lattice $\Lambda$. For example, when $d=2$, one can take $\Lambda=\{(m/2, n/2): m, n\in \Bbb Z\}$.  It is easy to see that $\{\Omega_i\}$ is a wavelet collection of sets.  Indeed, each pair $(\Omega_i, \Lambda^\perp)$ is a spectral pair and $\Omega$ is a multiplicative tiling set for $\Bbb R^d$ by the automorphisms $A= \{2^n I:  \ n\in \Bbb Z\}$,  where $I$ is the $d\times d$  identity matrix. The result now follows by  Theorem \ref{wavelet set-necessary and sufficient condition}. (For more examples of wavelet sets in  $\Bbb R^d$ see e.g. \cite{BB2009}.)\\
 \end{example} 
 
 \begin{example} Let $\Omega_i$ are given as in Example \ref{wavelet sets}. Assume that    $\Omega:=\cup_i \Omega_i$ is a  multiplicative  tiling with respect to the dyadic dilations $\{2^{n}I : n\in \Bbb Z\}$. 
 Take $\Lambda:= 4^{-1} \Bbb Z^d$. Then each pair $(\Omega_i, \Lambda)$ tiles $\Bbb R^d$ additively with multiplicity  $4$, i.e., almost every point in $\Bbb R^d$ is covered four times by  $\Lambda$ translations of each  $\Omega_i$.  Therefore $\Omega_i$ is a  Riesz spectral set with spectra  $\Lambda$. Theorem \ref{Riesz wavelet set} implies that $\{\Omega_i\}$ is a Riesz wavelet collection of sets  in  $\Bbb R^d$. 
 \\
 \end{example}

 \begin{example} 
 Let $\Omega_1$ be a symmetric polygon  in $\Bbb R^2$ centered at origin. Take $\Omega_2= 2^{-1} \Omega_1$, the dilation of $\Omega_1$ by scale $2^{-1}$. Define $\Omega= \Omega_1\setminus \Omega_2$. Then it is easy to check that $\Omega$ tiles the plane mutliplicatively with respect to the dyadic dilations $\{2^nI: n\in \Bbb Z\}$. 
 Assume that $\Omega_i, \cdots , \Omega_m$ be $m$ convex tiles of $\Omega$ which are pairwise disjoint. Furthermore, we assume that each $\Omega_i$ is a symmetric polygon with respect to a basis for $\Bbb R^2$. Since every convex and symmetric polygon in $\Bbb R^2$ has a Riesz spectra (\cite{Lyubarskii-Rashkovskii2000}), then  each $\Omega_i$ is a Riesz spectral set with spectra $\Lambda_i$.   Theorem  \ref{Riesz wavelet set} implies that   $\{\Omega_i\}_1^m$ is a Riesz wavelet collection of sets for $L^2(\Bbb R^2)$.   \\
 \end{example} 
 
 The following example illustrates  an approach to the construction of   an orthogonal wavelet basis  on $G$  induced by a wavelet set in a subgroup $K$.  
 
 \begin{example} 
Let $G$ be a locally  compact abelian group  which is topologically isomorphic to 
 $\Bbb R^n \times D\times K,$ where $D$     discrete   abelian group and $K$ is compact.  (Indeed, by  Theorem 24.30 of \cite{HwRoI},   every LCA group has this form.)  Assume that  $n$ is a non-negative integer and $D$ is a finite direct sum of finite abelian cyclic groups of prime power order, i.e., $p^r$, $p$ prime. Let $T:=\widehat K$ denote the  Pontryagin dual   group of $K$. $T$ is discrete and a spectra for $K$. 
 Therefore,  for given any spectral set $S$ in $\Bbb R^d$ with spectra $\Lambda$,   $\Omega:= S\times D\times K$ is  spectral   for $G$ with spectra $\Lambda\times D\times T$. 
 
Furthermore, assume that 
  $S$ tiles $\Bbb R^d$  multiplicatively with respect to a subset of automorphisms $A\subset GL(\Bbb R, d)$. For any $\alpha\in   A$ define $t_\alpha: G \to G$ with $t_\alpha(x,d,k) := (\alpha(x), k)$. Then $t_\alpha, \  \alpha\in A,$ is an automorphisms of $G$  and $\{t_\alpha(S\times D\times K): \alpha\in A\}$ is a mutual disjoint tiling (a partition) for $G$. This implies that  $\Omega$ is a wavelet set  for  $L^2(G)$.  \\
 \end{example}
  
 \iffalse 
 \begin{example} A locally compact abelian group $G$ is topologically isomorphic with 
  $\Bbb R^n \times H$ where $H$ is locally compact abelian group containing a compact open subgroup. (\cite{HwRoI}, Theorem 24.30). $H$ can be regarded as $\Bbb R^m\times \Bbb Z^l\times K$, where $K$ is a compact abelian group and $n,m$ and $l$ are non-negative integers. We assume that $G=
 \Bbb R^n \times D\times K,$ where $D$ is a  discrete abelian group and $K$ is compact abelian group and $n$ is a non-negative integer. We identify each $g\in G$ as $(x,k,d)$.   Let $n\geq 1$ and $\Omega_0\subset \Bbb R^n$ be a multiplicative  tiling set with respect to a subset  $\mathcal D\subset GL(d,\Bbb R)$.  For any $\alpha\in \mathcal D$ put $\tau_\alpha(x,d,k) := (\alpha(x), d, k)$. Then $\tau_\alpha: G\to G$ is an automorphism of $G$ and  $\Omega:=\Omega_0\times D\times K$ is a multiplicative tiling set for $G$ with respect to  $\{\tau_\alpha: \ \alpha\in \mathcal D\}$. Moreover, if $\Omega_0$ is (Reisz) spectral, then   $\Omega$ is a (Riesz)   spectral for $G$ and consequently a (Riesz) wavelet set in $G$.  
 \end{example} 
 \fi
% {\color{red} Example where the finite field $\Bbb F_q^d$ is also part of $G$.  }\\
 %{\color{red} Find the concepts of polygon and dilation for a LCA groups.}

\end{document}